\documentclass[12pt]{amsart}

\newtheorem{theorem}{Theorem}[section]

\numberwithin{equation}{section}

\begin{document}

\title[Curvature of symmetric product of Riemann surface]{On the
curvature of symmetric products of a compact Riemann surface}

\author[I. Biswas]{Indranil Biswas}

\address{School of Mathematics, Tata Institute of Fundamental
Research, Homi Bhabha Road, Bombay 400005, India}

\email{indranil@math.tifr.res.in}



\date{}

\begin{abstract}
Let $X$ be a compact connected Riemann surface of genus
at least two. The main theorem of \cite{BR} says that for any
positive integer $n\, \leq\, 2({\rm genus}(X)-1)$, the
symmetric product $S^n(X)$ does not admit any
K\"ahler metric satisfying the condition that all the holomorphic
bisectional curvatures are nonnegative. Our aim here is to give a very
simple and direct proof of this result of B\"okstedt and Rom\~ao.
\end{abstract}

\maketitle

\section{Introduction}\label{sec1}

The main theorem of \cite{BR} says the following (see \cite[Theorem 1.1]{BR}):

\textit{Let $X$ be a compact connected Riemann surface of genus
at least two. If
$$
1\, \leq\, n\, \leq\, 2({\rm genus}(X)-1)\, ,
$$
then the symmetric product $S^n(X)$ does not admit any K\"ahler metric
for which all the holomorphic bisectional curvatures are nonnegative.}

Our aim here is to give a simple proof of this theorem.

In \cite{Bi}, the following related was proved (see \cite[Theorem 1.1]{Bi}):

\textit{Let $X$ be a compact connected Riemann surface of genus
at least two. If
$$
n\, > \, 2({\rm genus}(X)-1)\, ,
$$
then $S^n(X)$ does not admit any K\"ahler metric for which all the
holomorphic bisectional curvatures are nonnegative.}

\section{Nonnegative holomorphic bisectional curvature}

Let $X$ be a compact connected Riemann surface of genus $g$,
with $g\, \geq\, 2$. For any positive integer $n$, the $n$-fold 
symmetric product of $X$ will be denoted by $S^n(X)$. We recall
that $S^n(X)$ is the quotient of $X^n$ by the natural action of the group
of permutations $P_n$ of the index set $\{1\, , \cdots \, ,n\}$. Let
\begin{equation}\label{q}
q\, :\, X^n\,\longrightarrow\, S^n(X)\,:=\, X^n/P_n
\end{equation}
be the quotient map.

\begin{theorem}\label{thm1}
Take any integer $n\, \in\, [1\, , 2(g-1)]$. The symmetric product $S^n(X)$
does not admit any K\"ahler metric satisfying the condition that all the
holomorphic bisectional curvatures are nonnegative.
\end{theorem}

\begin{proof}
Let $K^{-1}_{X^n} \,=\, \bigwedge\nolimits^n TX^n$ and $K^{-1}_{S^n(X)}
\,=\, \bigwedge\nolimits^n TS^n(X)$ be the anti-canonical line bundles
of $X^n$ and $S^n(X)$ respectively. Fix distinct $n-1$ points $\{x_1\, ,
\cdots\, ,x_{n-1}\}$ of $X$. Let
\begin{equation}\label{vp}
\varphi\, :\, X\, \longrightarrow\, X^n
\end{equation}
be the embedding defined by $x\, \longmapsto\, \{x\, ,x_1\, ,\cdots\, ,
x_{n-1}\}\,\in\, X^n$. Define
\begin{equation}\label{vpt}
\widetilde{\varphi}\,:=\, q\circ \varphi\, :\, X\, \longrightarrow\, S^n(X)\, ,
\end{equation}
where $q$ is defined in \eqref{q}.

We will show that
\begin{equation}\label{de}
\text{degree}(\widetilde{\varphi}^*K^{-1}_{S^n(X)})\,=\, n-2g+1\, .
\end{equation}

For $1\, \leq\, i\, <\, j\, \leq\, n$, let $D_{i,j}\, \subset\, X^n$
be the divisor consisting of all $\{y_1\, , \cdots\, , y_{n}\}$ such that
$y_i\, =\,y_j$. Let
$$
D\, :=\, \sum_{1 \leq i < j \leq n} D_{i,j}
$$
be the divisor on $X^n$. For the map $q$ in \eqref{q}, we have
$$
q^*K^{-1}_{S^n(X)}\,=\, K^{-1}_{X^n}\otimes {\mathcal O}_{X^n}(D)\, .
$$
Therefore,
\begin{equation}\label{d2}
\widetilde{\varphi}^*K^{-1}_{S^n(X)}\,=\, \varphi^* q^*K^{-1}_{S^n(X)}
\,=\, (\varphi^*K^{-1}_{X^n})\otimes (\varphi^*{\mathcal O}_{X^n}(D))\, ,
\end{equation}
where $\varphi$ is the map in \eqref{vp}.

Now,
$$
\varphi^*K^{-1}_{X^n}\,=\, TX ~\ \text{ and }\ ~
\varphi^*{\mathcal O}_{X^n}(D)\,=\, {\mathcal O}_X(\sum_{i=1}^{n-1} x_i)\, ,
$$
where $TX$ is the holomorphic tangent bundle of $X$. Therefore, from \eqref{d2}
we conclude that
$$
\widetilde{\varphi}^*K^{-1}_{S^n(X)}\,=\,(TX)\otimes
{\mathcal O}_X(\sum_{i=1}^{n-1} x_i)\, .
$$
Hence
$$
\text{degree}(\widetilde{\varphi}^*K^{-1}_{S^n(X)})\,=\,2-2g +n-1
\,=\, n-2g+1\, ,
$$
and \eqref{de} is proved.

Since $n\, \leq\, 2g-2$, from \eqref{de} we conclude that
$$
\text{degree}(\widetilde{\varphi}^*K^{-1}_{S^n(X)})\, <\, 0\, .
$$
This immediately implies that $S^n(X)$ does not admit a K\"ahler metric
for which all the holomorphic bisectional curvatures are
nonnegative.
\end{proof}

The Chern class $c_1(K^{-1}_{S^n(X)})$ is computed in \cite[p. 333]{Ma}
and \cite[p. 323]{ACGH}. It should be possible to derive
\eqref{de} using this description of $c_1(K^{-1}_{S^n(X)})$.

%%%%%%%%%%%%%%%%%%%%%%%%%%%%%%%%%%%%%%%%%%%%%%%%%%%%%%%%%%%%%%%%% 

\end{document}